\documentclass{amsart}[12pt]
\usepackage{amssymb, amsfonts}
\usepackage{hyperref}

\usepackage[retainorgcmds]{IEEEtrantools}

\newtheorem{theorem}{Theorem}[section]
\newtheorem{lemma}[theorem]{Lemma}
\newtheorem{proposition}[theorem]{Proposition}

\theoremstyle{definition}

\newtheorem{notation}[theorem]{Notation}

\newcommand{\Ad}{\mathop{\mathrm{Ad}}\nolimits}
\newcommand{\ad}{\mathop{\mathrm{ad}}\nolimits}

\newcommand{\Kil}{\mathop{\mathrm{Kil}}\nolimits}

\newcommand{\cA}{{\mathcal A}}
\newcommand{\cB}{{\mathcal B}}

\newcommand{\cH}{{\mathcal H}}

\newcommand{\cO}{{\mathcal O}}

\newcommand{\cT}{{\mathcal T}}

\newcommand{\ccH}{{\mathcal H^\bC}}
\newcommand{\cfg}{{\mathfrak g^\bC}}
\newcommand{\cfh}{{\mathfrak h^\bC}}
\newcommand{\cfk}{{\mathfrak k^\bC}}
\newcommand{\cfm}{{\mathfrak m^\bC}}

\newcommand{\xil}{{\xi_\l}}

\newcommand{\bR}{{\mathbb R}}
\newcommand{\bC}{{\mathbb C}}

\newcommand{\fg}{{\mathfrak g}}
\newcommand{\fh}{{\mathfrak h}}
\newcommand{\fk}{{\mathfrak k}}
\newcommand{\ft}{{\mathfrak t}}
\newcommand{\fm}{{\mathfrak m}}

\newcommand{\<}{\langle}
\renewcommand{\>}{\rangle}

\renewcommand{\a}{\alpha}
\renewcommand{\b}{\beta}

\renewcommand{\o}{\omega}
\renewcommand{\O}{\Omega}

\newcommand{\D}{\Delta}

\renewcommand{\l}{\lambda}
\renewcommand{\d}{\delta}
\renewcommand{\th}{\theta}

\newcommand{\dia}{\diamond}

\newcommand{\zd}{Z_\diamond}

\numberwithin{equation}{section}

\begin{document}

\title[Cotangent bundles]{Cotangent bundles for ``matrix algebras converge to the sphere''}

\author{Marc A. Rieffel}
\address{Department of Mathematics\\
University of California\\
Berkeley, CA\ \ 94720-3840}
\curraddr{}
\email{rieffel@math.berkeley.edu}
\thanks{This work is part of the project sponsored by European Union 
grant number H2020-MSCA-RISE-2015-691246-QUANTUM DYNAMICS}
\dedicatory{Dedicated to the memory of Richard V. Kadison}


\subjclass[2000]{Primary 53C30; Secondary 46L87, 58J60, 
53C05}
\keywords{cotangent bundles, matrix algebras, differential calculi, compact Lie groups, ergodic actions, coadjoint orbits}

\large{

\begin{abstract}
In the high-energy quantum-physics literature one finds statements such as ``matrix
algebras converge to the sphere''. Earlier I provided a general setting for understanding such
statements, in which the matrix algebras are viewed as compact quantum metric spaces, and
convergence is with respect to a quantum Gromov-Hausdorff-type distance. More recently I
have dealt with corresponding statements in the literature about vector bundles on spheres and
matrix algebras.
But physicists want, even more, to treat structures on spheres (and other spaces) such as Dirac
operators, Yang-Mills functionals, etc., and they want to approximate these by corresponding
structures on matrix algebras. In preparation for understanding what the Dirac operators should be,
we determine here what the corresponding "cotangent bundles" should be for the matrix algebras, 
since it is on them that a "Riemannian metric" must be defined, 
which is then the information needed to determine a Dirac operator. 
(In the physics literature there are at least 3 inequivalent suggestions for the Dirac operators.) 
\end{abstract}

\maketitle

\section*{Introduction}

In the literature of theoretical high-energy physics one finds statements along the lines of ``matrix algebras converge to the sphere'' and ``here are the Dirac operators on the matrix algebras that correspond to the Dirac operator on the sphere''.  But one also finds that at least three inequivalent types of Dirac operator are being proposed in this context.  See, for example, \cite{AIM,Aok,BIm,BKV,CWW,GP2,HQT,Yd1,Ydr} and the references they contain, as well as \cite{Ngo} which contains some useful comparisons.  In \cite{R6,R7,R21, R29} I provided definitions and theorems that give a precise meaning to the convergence of matrix algebras to spheres.  These results were developed in the general context of coadjoint orbits of compact Lie groups, which is the appropriate context for this topic, as is clear from the physics literature. I seek  to give eventually a precise meaning to the statements about Dirac operators.  

In ordinary differential geometry, Dirac operators are
built from Riemannian metrics, which give a smooth 
assignment of an inner
product to the tangent vector space at each point of the manifold.
But in the non-commutative setting suitable ``tangent bundles''
are scarce, while ``cotangent bundles'' are relatively common.
They are often called ``first order differential calculi'' \cite{GVF}.
 In ordinary differential geometry it is well-known that
 a Riemannian metric can equivalently be specified by
 giving a smooth assignment of an inner
 product to the cotangent vector space (the dual
 of the tangent vector space) at each point of the manifold.
The main result of this paper is to indicate what the ``cotangent
 bundles'' are for the matrix algebras that converge to the
 sphere and to other spaces. The appropriate context
 is that of connected compact semisimple Lie groups, 
 and that is the context in which we work in this paper.
 The statement and proof require the detailed theory
 of roots and weights for semisimple Lie groups and
 their representations, and we prefer to state our main
 result (Theorem \ref{mainthm}) after we have 
 established our notation and conventions 
 for this detailed theory. The particular case in which
 $G = SU(n)$ with its defining representation
 of $G$ on $\bC^n$ was treated earlier in
 \cite{Dbv, DKM, Msn}.
 
In the non-commutative context the ``cotangent
 bundles'' are actually bimodules, which in the
 commutative context are the bimodules of
 smooth cross-sections for the ordinary cotangent
 bundles. In the non-commutative context we will
 continue to refer to these bimodules as ``cotangent
 bundles".
 
 After passing my qualifying exam I went to talk with Dick Kadison
 about possible research directions. He suggested that I think 
about the relations between groups and operator algebras. 
This paper is one bit of the evidence that I have been 
following his suggestion ever since, with great pleasure.
 
 
 \tableofcontents


\section{Preliminaries on compact Lie groups and their representations}
\label{secpre}

Let $T$ be a torus group, that is, a commutative connected compact Lie group,
isomorphic to a finite product of circle groups. We will denote
its Lie algebra by the traditional $\fh$. For any finite-dimensional unitary representation $(\cH, \pi)$
of $T$ we let $\pi$ also denote the corresponding representation of $\fh$. For each $H \in \fh$
the operator $\pi_H$ is skew-adjoint, and so its eigenvalues are purely imaginary. Since
the $\pi_H$'s all commute with each other, they are simultaneously diagonalizable. Because
we need to keep track of the structure over $\bR$, we will use a convention for the weights of
a representation that is slightly different from the usual convention. If $\xi \in \cH$ is a common
eigenvector for the $\pi_H$'s, there will be a linear functional $\a$ on $\fh$ (with values in $\bR$) 
such that 
\[
\pi_H(\xi) = i\a(H)\xi
\]
for all $H \in \fh$.
For each $\a \in \fh'$ (where $\fh'$ denotes the dual vector space to $\fh$) we set
\[
\cH_\a = \{ \xi \in \cH:\pi_H(\xi) = i\a(H) \xi \ \ \mathrm{for \ all} \  H \in \fh\}  .
\]
If there are non-zero vectors in $\cH_\a$ then we say that $\a$ is a \emph{weight} of the
representation $(\cH, \pi)$. We denote the set of all weights for this representation by
$\D_\pi$. Then
\[
\cH = \bigoplus \{\cH_\a : \a \in \D_\pi \}    .
\]

Suppose, instead, that $\cH$ is a Hilbert space over $\bR$ and that $\pi$ is a representation
of $T$ by orthogonal transformations. 
The corresponding representation of $\fh$ is by
skew-symmetric operators, which may have no eigenvectors. 
Let $\ccH$ denote the 
complexification of $\cH$, and let $\th$ denote the corresponding complex conjugation
operator on $\ccH$, so that $\th$ is a conjugate linear 
isometry such that $\th^2 = I_\ccH$. 
Let $\pi$ also denote the extension of $\pi$ to $\ccH$. 
Notice that $\th$ commutes with each $\pi_H$. 
Let $\a$
be a weight of $\pi$, and let $\xi  \in \cH_\a^\bC$. 
Then for any $H \in \fh$
\[
\pi_H(\th \xi) = \th(\pi_H \xi) = \th(i\a(H)\xi) = -i\a(H)\th(\xi).
\]
Thus $\th$ carries $\cH_\a^\bC$ into, in fact onto, $\cH_{-\a}^\bC$. 
Thus when a unitary
representation is the complexification of an orthogonal representation, 
if $\a$ is a
weight of the representation then so is $-\a$. Let $v = \xi + \th(\xi)$ and 
$w = i(\xi - \th(\xi))$, so that $v, w \in \cH$. Then $\pi_H(v) = i\a(H)w$ and
$\pi_H(w) = -\a(H)v$. 

Now let $G$ be a compact connected semisimple Lie group. For discussion
and proofs of the results we state below see \cite{BrD, Knp, Srr, Smn} .
We make a choice of
a maximal torus, $T$, in $G$. Let $\fg$ denote the Lie algebra of $G$, and let $\fh$
be its subalgebra for $T$. As in \cite{R27}, we let $\Kil$ denote the negative of the
Killing form on $\fg$, so that it is a (positive) inner product on $\fg$. Then the adjoint
representation, $\Ad$, of $G$ on $\fg$ is by orthogonal operators for $\Kil$. Thus
the corresponding adjoint representation, $\ad$, of $\fg$, which is just the left
regular representation of $\fg$ on itself, is by skew-symmetric operators for $\Kil$.
We let $\fg^\bC$ denote the complexification of $\fg$.
The non-zero weights for $\Ad$ or $\ad$ are called the ``roots'' of $G$. 
We denote the set
of roots simply by $\D$. By the comments made above, if $\a \in \D$ then
$-\a \in \D$. In the standard way \cite{Knp, Srr, Smn} we make a choice, $\D^+$, of  
positive roots, and we let $S$ denote the corresponding 
set of simple roots in $\D^+$.
For each root $\a$ we let $\fg^\bC_\a$ denote the corresponding root space. 
We extend $\Kil$ to $\cfg$ by $\bC$-bilinearity (not sesquilinearity).
It is a standard fact that this extended $\Kil$ is non-degenerate, and that the
root spaces $\fg^\bC_\a$ and $\fg^\bC_\b$ are orthogonal to each other for $\Kil$
exactly if $\a-\b \neq 0$, while all root spaces are orthogonal
to $\cfh$. It is also
a standard fact that these root spaces are all of dimension 1, and 
that $[\fg^\bC_\a, \fg^\bC_{-\a}]$ is not of dimension 0 (so is of dimension 1). 
We want to choose usual elements $H_\a$, $E_\a$ and $F_\a$ in these spaces, 
but we need to choose them in a careful way 
so that they mesh well with representations.

Let $(\cH, \pi)$ be a finite-dimensional unitary 
representation of $G$. We extend the corresponding
representation of $\fg$  to a representation (still denoted by $\pi$) of $\cfg$.
Let $W \in \cfg$ with $W = X + iY$ for $X, Y \in \fg$. Then
\[
(\pi_W)^* = (\pi_X)^* + (i\pi_Y)^* = \pi_{(-X + iY)}    .
\]
Thus it is appropriate to define an involution on $\cfg$ by $(X+iY)^* = -X + iY$,
so that $(\pi_W)^* = \pi_{W^*}$ for all $W \in \cfg$ (as in \cite{Hll, Smn}). 
Notice that for all $W, Z \in \cfg$
we have $[W, Z]^* = [Z^*, W^*]$.

The following result is certainly well-known, but I have not seen in the literature a
derivation of it quite like the one below, though there are similarities with results in
\cite{Hll, Smn}. 

\begin{proposition}
\label{HEF}
With notation as above, for each $\a \in \D^+$ we can choose $H_\a \in i\fh$ and
$E_\a \in \fg^\bC_\a$  such that $[E_\a, E_\a^*] = H_\a$ and
$[H_\a, E_\a] = 2E_\a$. Setting
$F_\a = E_\a^*$, we then obtain $[H_\a, F_\a] = -2F_\a$.
\end{proposition}

\begin{proof}
 Let $\a \in \D$ be given, and choose 
a non-zero $E \in \fg^\bC_\a$. Then $E^* \in \fg^\bC_{-\a}$ and $E^* \neq 0$.
Then $[E, E^*] \neq 0$ since $[\fg^\bC_\a, \fg^\bC_{-\a}]$ is not of dimension 0.
Furthermore, $[E,E^*]$ is self-adjoint for $^*$, and so is in $i\fh$.
We must relate all this to $\Kil$. For any $H \in \fh$ we have
\[
\Kil(H, [E,E^*]) = \Kil(-[E,H], E^*) = i\a(H) \Kil(E, E^*).
\]
It is easily calculated that $\Kil(Z,Z^*)$ is strictly negative for any non-zero
$Z \in \cfg$. Rescale $E$ so that $\Kil(E,E^*) = -1$. Then for all $H \in \fh$
\[
\a(H) = \Kil(H, i[E,E^*])  .
\] 
Set $\tilde H_\a =  [E, E^*]$. Then
\[
[\tilde H_\a, E] = i\a(\tilde H_\a)E = -\Kil(\tilde H_\a, \tilde H_\a)E   .
\] 
Notice that the coefficient of $E$ on the right side is positive. This equation
says that 
\[
[[E,E^*],E] = -\Kil([E,E^*],[E,E^*])E.
\]
It is then clear that we can rescale $E$ so  that the coefficient of $E$ on 
the right side has value 2.
Denote the resulting $E$ by $E_\a$ and set $H_\a = [E_\a,E_\a^*]$. We see
that $[H_\a, E_\a] = 2E_\a$ as desired.
\end{proof}


\section{Highest weight vectors}
\label{sechigh}

Let $(\cH, \pi)$ be an irreducible unitary representation of $G$.
By the standard theory \cite{Knp, Srr, Smn}, for our choice of
$\D^+$ made in the previous section there is a highest weight
vector, $\xi_o \in \cH$, for $\pi$, with $\|\xi_o\| = 1$. It is unique up
to phase. As a weight vector it is an eigenvector for all the $\pi_H$
for $H \in \fh$. The fact that it is a highest weight vector means
exactly that $\pi_{E_\a}\xi_o = 0$ for all $\a \in \D^+$. Define $\l$ on
$\fg$ by
\[
\l(X) = -i\<\xi_o, \pi_X\xi_o\>     .
\]
(We take the inner product on $\cH$ to be linear in the
second variable, as done in \cite{R27, Hll, GVF}.)
Up to sign $\l$ is exactly the ``equivariant momentum map'' of
equation 23 of \cite{Ln2} evaluated on the highest weight vector.
Because $\pi_X$ is skew-symmetric for all $X \in \fg$, we
see that $\l$ is $\bR$-valued on $\fg$.
Extend $\l$ to $\cfg$ in the usual way. 
Notice that for any $\a \in \D^+$ we have
\[
i\l(H_\a)=\<\xi_o, [\pi_{E_\a}, \pi_{E_\a^*}]\xi_o\> =
\<\xi_o, \pi_{E_\a} \pi_{E_\a^*}\xi_o\> \geq 0 ,
\]
so that $\l$ is "dominant".
Note that $\l$ does not depend on the phase of $\xi_o$.
From now on we
will denote $\xi_o$ by $\xi_\l$.

 Because $\xi_\l$ is a highest weight vector, 
we clearly have
$\l(E_\a) = 0$ for all $\a \in \D^+$, and $\l(F_\a) = 0$ for 
all $\a \in \D^+$ because $F_\a = E_\a^*$. Furthermore, because
$[E_\a, E_a^*] = H_\a$ and
$[H_\a, E_\a] = 2E_\a$ and $[H_\a, F_\a] = -2F_\a$, the
triplet $(H_\a, E_\a, F_\a)$ generates via $\pi$ a representation 
of $sl(2,\bC)$, 
for which the spectrum of $\pi_{H_\a}$ must consist
of integers. In particular, $i\l(H_\a)$ is an integer, necessarily
non-negative, in fact equal to $\|F_\a\xil\|^2$. 
We see in this way that $\l$ is a quite special
element of $\fg'$.

Let $\mu$ denote the weight of $\xi_\l$, so that 
$\pi_H(\xi_\l) = i\mu(H)\xi_\l$ for all $H \in \fh$. 
Comparison with the definition of $\l$ shows that $\mu$ is
simply the restriction of $\l$ to $\fh$.  It is clear
that $\l$ is determined by $\mu$ in the sense that $\l$
has value $0$ on the $\Kil$-orthogonal complement of $\cfh$.
Thus from now on
we will let $\l$ also denote the weight of $\xi_\l$.
(Thus the special properties of $\l$ mean that,
as a weight, $\l$ is a ``dominant integral weight''.)

\section{Coadjoint orbits}
\label{seccoad}
Let $\mu \in \fg'$ with $\mu \neq 0$. The coadjoint orbit of $\mu$
is $\cO_\mu = \{\Ad_x'(\mu): x \in G\}$. Then $G$ acts transitively 
on $\cO_\mu$. Let $K=\{x \in G: \Ad_x'(\mu) = \mu\}$, the stability 
subgroup of $\mu$. Then $\cO_\mu$ can be naturally identified
with the homogeneous space $G/K$. As in \cite{R27} we will
usually work with $G/K$ rather than directly with $\cO_\mu$.
Let $\fk$ be the Lie algebra of $K$. Then it is evident that
$\fk = \{Y \in \fg: \mu([Y,X]) = 0 \ \ \mathrm{for \ all} \ \ X \in \fg\}$.

Since $\Kil$ is definite on $\fg$, there is a (unique) element in $\fg$,
denoted by $\zd$ in \cite{R27}, such that
\[
\l(X) = \Kil(X, \zd)
\]
for all $X \in \fg$.
It is easily seen that the $\Ad$-stability subgroup of $Z_{\dia}$ is again $K$.
Let $T_{\dia}$ be the closure in $G$ of the one-parameter 
group $r \mapsto \exp(rZ_{\dia})$, so that $T_{\dia}$ is a torus subgroup of $G$.  
Then it is easily seen that $K$ consists exactly of all the elements of $G$ that 
commute with all the elements of $T_{\dia}$.  Note that $T_{\dia}$ is contained 
in the center of $K$ (but need not coincide with the center).  Since each element 
of $K$ will lie in a torus subgroup of $G$ that contains $T_{\dia}$, it follows 
that $K$ is the union of the tori that it contains, and so $K$ is 
connected (corollary 4.22 of  \cite{Knp}).  Thus for most purposes we can just work 
with the Lie algebra, $\fk$, of $K$ when convenient.  In 
particular, $\fk = \{X \in \fg: [X,Z_{\dia}]  = 0\}$, and $\fk$ contains 
the Lie algebra, $\ft_{\dia}$, of $T_{\dia}$.

Let us apply the above considerations to the $\l$ of the previous section. We 
view $\l$ as extended to $\cfg$. We saw that for all $\a\in \D^+$ we have
$\l(E_\a) = 0 = \l(F_\a)$. It follows that $\zd$ is $\Kil$-orthogonal to
all the root spaces of $\cfg$, and so is in $\cfh$. But also $\zd \in \fg$, and
so $\zd \in \fh$. It follows that $T_{\dia}$ is contained in the maximal
torus $T$ that we had chosen in the previous section. But $K$ is the
centralizer of $T_{\dia}$, and so $K$ contains $T$. Consequently
$\fh \subseteq \fk$ and $\cfh \subseteq \cfk$.

As in \cite{R27} let $\fm = \fk^\perp$ (for $\Kil$). As seen there (and in many
other places), $\fm$ is naturally identified with the tangent space
at the coset $K$ of $G/K$, and we will use this later. We have
further that $\cfm = \cfk^\perp$. 
We now make more precise for our special situation some results in 
section 3 of \cite{BFR}.

\begin{proposition}
\label{struk}
With notation as above, $\cfk$ is the direct sum of $\cfh$ with
the span of $\{E_\a, F_\a: \l(H_\a) = 0\}$, while $\cfm$ is
the span of $\{E_\a, F_\a: \l(H_\a) \neq 0\}$.
\end{proposition}
\begin{proof}
We saw above that $\cfh \subseteq \cfk$. Since $T$ is clearly
a maximal torus in $K$ it follows that $\cfk$ is the direct sum
of $\cfh$ and the weight spaces that it contains. The proof 
of the first statement is
then completed by:

\begin{lemma}
\label{lemk}
Let $\a \in \D^+$. If $\l(H_\a) = 0$ then $E_\a, F_\a \in \cfk$. Conversely,
if either $E_\a \in \cfk$ or $F_\a \in \cfk$ then $\l(H_\a) = 0$.
\end{lemma}
\begin{proof}
 If $\l(H_\a) = 0$ then $\pi_{H_\a}\xi_\l = \l(H_\a)\xi_\l = 0$. Also
 $\pi_{E_\a}\xi_\l = 0$ since $\xi_\l$ is a highest weight vector.
 Since $\{H_\a, E_\a, F_\a\}$ generate a representation of $sl(2, \bC)$
 with the usual relations, the facts about such representations 
 (see \cite{Knp, Srr, Hll}) imply
 that $\pi_{F_\a}\xil = 0$. But then for any $X\in \cfg$ we have
 \[
 \l([E_\a, X]) = \<\pi_X\xil, \pi_{F_\a}\xil\> - \<\pi_X\pi_{E_\a}\xil, \xil\> = 0,
 \]
 so that $E_\a \in \cfk$. A similar argument shows that $F_\a \in \cfk$.
 Conversely, if $E_\a \in \cfk$ then $\<[E_\a, X]\xil, \xil\> = 0$ for 
 any $X \in \cfg$. On setting $X = F_\a$ we find 
 that $\l(H_\a) = \<[E_\a, F_\a]\xil, \xil\> = 0$. A similar argument applies
 if it is $F_\a$ that is in $\cfk$.
\end{proof}
We return to the proof of Proposition \ref{struk}
Suppose that $\l(H_\a) \neq 0$. Then for every $\b \in \D$ 
such that $\l(H_\b)=0$ we have $\a-\b \neq 0$ and
so $E_\a$ and $F_\a$ are orthogonal to $E_\b$ and $F_\b$.
Thus $E_\a$ and $F_\a$ are orthogonal to $\cfk$. From this
the second statement follows quickly.
\end{proof}

\section{The cotangent bundles for the matrix algebras}
\label{cotanb}

With notation as used earlier, we let $\cB = \cB(\cH_\l)$, and we let
$\a$ be the action of $G$ on $\cB$ defined 
by $\a_x(T) = \pi_x T \pi_x^*$. The corresponding representation
of $\fg$ is given by $\a_X(T) = [\pi_X, T]$. As a first approximation
to the cotangent bundle we take $\cB\otimes \fg'$ ($=\cB\otimes (\cfg)'$), 
viewed as a
$\cB$-bimodule in the evident way. For any $T \in \cB$ we
define $dT$ by $(dT)(X) = \a_X(T) = [\pi_X, T]$. Then $d$
is a derivation of $\cB$ into the bimodule $\cB\otimes \fg'$.
But the definition of the cotangent bundle (or first order calculus \cite{GVF})
includes the requirement
that it be generated as a bimodule by the range of $d$. So our task
is to determine for our situation what this sub-bimodule 
of  $\cB\otimes \fg'$ is.

The representation $(\cH_\l, \pi)$ need not be faithful. Its kernel
at the Lie-algebra level is an ideal of $\fg$. But $\fg$, as a semisimple
Lie algebra, is the direct sum of its minimal ideals, each of which
is a simple Lie algebra (non-commutative). 
Denote the kernel of $\pi$ by $\fg_o$. It must be the direct sum of
some of these minimal ideals. Denote the direct sum of the remaining
minimal ideals by $\fg_\l$, so that $\fg = \fg_\l \oplus \fg_o$. Clearly
$\pi$ is faithful on $\fg_\l$. We identify $\fg_\l'$ with the subspace
of $\fg'$ consisting of linear functionals on $\fg$ that take value 0
on $\fg_o$.

From the definition of $dT$ it is clear that $(dT)(X)$ is 0 for any $X$ in $\fg_o$.
Consequently, the range of $d$ is contained in  the $\cB$-bimodule
$\cB\otimes \fg_\l'$. The main theorem of this section, and of this paper,
is:

\begin{theorem}
\label{mainthm}
With notation as above, the $\cB$-bimodule generated by the range of
$d$ is $\cB\otimes \fg_\l'$. Thus $\cB\otimes \fg_\l'$ is the cotangent bundle
for $\cB$ for the action $\a$.
\end{theorem}

\begin{proof}
It is clear from the discussion above that it is sufficient to prove that if
$\pi$ is a faithful representation of $\fg$ then the $\cB$-bimodule generated
by the range of $d$ is $\cB\otimes \fg'$. Thus we assume that $\pi$ is
faithful for the rest of the proof.

For notational simplicity, in the 
rest of the proof we will use module notation
for the action of $\fg$ on $\cH_\l$, not mentioning $\pi$. Thus we will write
$X\eta$ for $\pi_X(\eta)$, for example. 

Let $\O_\l$ be the linear span of all the linear functionals from $\fg$
into $\cB$ of the form
\[
X \mapsto R(dT(X))S
\]
for $R, S, T\in \cB$. Clearly from the definition, 
$\O_\l$ is the cotangent bundle
that we seek. Thus our task is to show that $\O_\l = \cB\otimes \fg'$.
Now every operator in $\cB$ is the sum of rank-one operators. Thus
$\O_\l$ is the linear span of the functionals of the above form for which
$R$ and $S$ are of rank one. For the purpose of examining these
operators we use the following notation. For $\xi, \eta \in \cH_\l$
we let $\<\xi, \eta\>_o$ denote the rank-one operator defined by
\[
\<\xi, \eta\>_o(\zeta) = \xi\<\eta, \zeta\>    
\] 
for $\zeta \in \cH_\l$, where the inner product on the right side 
is that of $\cH_\l$ (assumed linear in its second variable). 
Thus for $\xi, \eta, \zeta, \o \in \cH_\l$ and
for $T \in \cB$ we consider linear functionals from $\fg$ into
$\cB$ of the form
\begin{align*}
 X \mapsto & \<\xi, \eta\>_o [T, X]\<\zeta, \o\>_o  
 = \<\<\xi, \eta\>_o [T, X]\zeta, \o\>_o  \\
 &= \<\xi\<\eta, [T, X]\zeta\>, \o\>_o   
= \<\eta, [T, X]\zeta\>\<\xi, \o\>_o      .
\end{align*} 
Fixing $\eta, \zeta$ and $T$ and taking linear combinations
for various $\xi, \o$, we see that we obtain in this way all
of $\<\eta, [T, X]\zeta\> \cB$. So we see that it is sufficient
for us to consider linear combinations of linear functionals 
of the form
\[
X \mapsto \<\eta, [T, X]\zeta\>   .
\]
We denote the linear span of such functionals by $Q_\l$,
and we see that our task is to show that $Q_\l = (\cfg)'$.
Now each of $\eta$ and $\zeta$ is a linear combination of
weight vectors, and so it suffices for us to examine the 
case in which $\eta$ and $\zeta$ are weight vectors.
Thus, if $\mu$ and $\nu$ are weights and if $\xi_\mu$ and 
$\xi_\nu$ are weight vectors for them, it suffices to consider
functionals of the form
\[
X \mapsto \<\xi_\mu, [T, X]\xi_\nu\>. 
\]

Let $\a \in \D^+$ be given. Since $\pi$ is faithful and weight vectors
span $\cH_\l$, there is a weight vector, $\xi_\mu$, such that
$F_\a \xi_\mu \neq 0$. Then the representation of the 
$sl(2)$-subalgebra spanned by $\{H_\a, E_\a, F_\a\}$
generated by $\xi_\mu$ has dimension at least 2. We can 
change $\xi_\mu$ to be a highest weight vector for this
$sl(2)$-representation. Then $i\mu(H_\a) > 0$, while
$E_\a \xi_\mu = 0$ and $F_\a \xi_\mu \neq 0$.

Let $\phi$ be the linear functional on $\cfg$ defined by
\[
\phi(X) = \<\xi_\mu, [E_{\a}, X]F_{\a}\xi_\mu\>    .
\]
If $H \in \cfh$, then
\[
\phi(H) = -i\a(H)\<\xi_\mu, E_{\a}F_{\a}\xi_\mu\>   
 =-i\a(H)\<F_{\a}\xi_\mu, F_{\a}\xi_\mu\>   ,
\]
which is a non-zero multiple of $\a(H)$ 
since $F_{\a}\xi_\mu \neq 0$. On the other hand, if $X = E_\b$ 
or $X = F_\b$ for some $\b \in \D$ then $\phi(X)=0$ because weight
vectors for different weights are orthogonal. We see 
in this way that $Q_\l$ contains all linear
functionals on $\cfg$ that take value 0 on
the $Kil$-orthogonal complement of $\cfh$.

Now let us define $\phi$ instead by
\[
\phi(X) = \<\xi_\mu, [E_\a, X]\xi_\mu\>   .
\]
By considering the weights of the vectors involved, it is immediate that
$\phi(H) = 0$ for all $H \in \cfh$, and that $\phi(E_\b) = 0$ for all
$\b \in \D^+$.
Furthermore, by similar considerations, $\phi(F_\b)= 0$ if $\b \neq \a$,
while $\phi(F_\a)$ is a non-zero multiple of $\mu(H_\a)$ ($\neq 0$), 
so that $\phi(F_\a) \neq 0$. So we see that $Q_\l$ contains a 
non-zero linear
functional that is 0 on the $Kil$-orthogonal complement of $F_\a$. 
By replacing $F_\a$
by $E_\a$ in the formula for $\phi$, one finds in the same way that
$Q_\l$ contains a non-zero linear
functional that is 0 on $(E_\a)^\perp$. 
Putting all of this together, we see that $Q_\l = (\cfg)'$, as desired.
\end{proof}


\section{The cotangent bundle for G}
\label{cotang}
In this short section, as a prelude to discussing the cotangent
bundle for coadjoint orbits, we examine the cotangent bundle
for $G$. Here we only need to assume that $G$ is a connected
compact Lie group, with Lie algebra $\fg$. In this section we
will not need to take compexifications of $\fg$ and other 
vector spaces, so all vector spaces will be over $\bR$.

We 
let $\cA = C^\infty(G)$, and we let $\a$ denote the action of
$G$ on $\cA$ by left translation. We let $\a$ also denote the
corresponding action of $\fg$ on $\cA$. According to 
our consistent approach to cotangent bundles, we first
consider the $\cA$-bimodule 
$\cA\otimes {\fg}' = C^\infty(G, \fg')$, and the
derivation $d$ into it defined by 
$df(X) = \a_X(f)$ for $f \in \cA$ and $X \in \fg$. The
cotangent bundle is then
the sub-$\cA$-bimodule generated by
the range of $d$. Since it is well-known that for the usual
definition of cotangent bundles the fibers of the usual
cotangent bundle of $G$ are just copies of ${\fg}'$, it is
no surprise that we have:

\begin{theorem}
\label{thmcotang}
For notation as above, the cotangent bundle for $G$, i.e.
for $\cA$, is $\cA\otimes {\fg}'$ itself.
\end{theorem}

\begin{proof}
Let $\{X_j\}_{j = 1}^n$ be a basis for $\fg$ (so the dimension of
$\fg$ is $n$). For any fixed $r\in\bR, r >0$ let $C_r$ denote the
 open hypercube $(-r,r)^n$. Let $\mathrm{exp}$ be the exponential
 map from $\fg$ into $G$, and let $\Phi:C_r \to G$ be defined by
 $\Phi(t_1, \cdots, t_n) = \mathrm{exp}(t_1X_1 + \cdots +t_nX_n)$.
 Choose $r$ sufficiently small that $\Phi$ is a diffeomorphism
 from $C_r$ onto an open neighborhood of the identity element of $G$.
 For each $j$ let $x_j$ denote the standard coordinate function
 on $C_r$. The differentials $dx_j$ form a basis for the 
 $C^\infty(C_r)$-bimodule of smooth cross-sections 
 of the usual cotangent bundle, i.e differential forms. 
 
 Any 1-form $\o$
 of compact support on $C_r$ can be expressed as a linear
 combination of the $dx_j$'s with coefficients in $C^\infty_c(C_r)$.
 Since $\o$ has compact support in $C_r$, we can find
 a smooth function, $h$, on $C_r$
that takes value 1 on the support of $\o$ but has compact
support inside $C_r$. For each $j$ let $h_j = hx_j$.
Then $\o$ can be expressed as a linear
 combination of the $dh_j$'s with coefficients 
 in $C^\infty_c(C_r)$.
 Since $\Phi$ is a diffeomorphism, this picture carries over to
 $\Phi(C_r)$, so any 1-form on $G$ with compact support in
 $\Phi(C_r)$ will be a linear
 combination of the images of the $dh_j$'s with coefficients 
 in $C^\infty_c(\Phi(C_r))$.
 Extending the images of the $h_j$'s 
 and the coefficients to functions 
 in $C^\infty(G)$ that take value 0 outside $\Phi(C_r)$, 
 we see that any 1-form on $G$ with support in $\Phi(C_r)$
 is in the bimodule
generated by the range of $d$.
We can cover $G$ by a finite number of translates
 of $\Phi(C_r)$, and then find a smooth partition of the
 identity, $\{p_k\}$, subordinate to this cover. Given a 1-form $\o$
 on $G$, each of the $p_k\o$'s will
 be in the $\cA$-bimodule generated by the range of $d$, and thus
 $\o$ itself will be in that bimodule, as needed.
\end{proof}
The situation for homogenous spaces, in particular for coadjoint orbits,
is more complicated.


\section{Cotangent bundles for homogeneous spaces}
\label{cothom}
In this section we treat the cotangent bundle for homogeneous
spaces $G/K$ where $G$ is now any compact connected Lie
group, and $K$ is any closed connected subgroup of $G$.
In this paper we are primarily interested in the case in which 
$G$ is semisimple and $K$ is the stability subgroup for a
point in a coadjoint orbit for $G$. But for just the construction
of the cotangent bundle nothing special happens for that
more special situation. What is special in that situation
is that then the coadjoint orbit has a Kahler structure. That
is important when constructing a corresponding Dirac
operator, as seen in \cite{R27}, but we will not discuss
that aspect in this paper. 

In this section we will not need to complexify the Lie algebras,
and so again all vector spaces will be over $\bR$. As in the earlier
sections, $\fg$ and $\fk$ will denote the Lie algebras of $G$ and
$K$. A description of the (smooth cross-sections of the) 
tangent bundle was given in \cite{R22}. We will make
use of that description here. As is frequently done in the
present situation, we choose and fix 
an $Ad$-invariant inner product on $\fg$. 
(When $G$ is semisimple it can be
our earlier $\Kil$.) Much as done earlier, we set $\fm = \fk^\perp$.

In notation 4.2 of \cite{R22} the tangent bundle
of $G/K$ was described as
\[
\cT(G/K) = \{W \in C^{\infty}(G,\fm): W(xs) = 
\Ad_{s^{-1}}(W(x)) \mbox{ for } x \in G,\ s \in K\}   .
\]
For this definition, elements of $\cT(G/K)$ act as derivations on 
$\cA =C^{\infty}(G/K)$ by
\[
(\delta_Wf)(x) = D_0^t(f(x\exp(tW(x)))),
\]
where we write $D_0^t$ for 
$(d/dt)|_{t=0}$.
Notice that this definition of $\delta_W$ involves right 
multiplication even though we have usually used
left multiplication. Reasons for using right multiplication
here are given in \cite{R22}.
It is 
clear that $\cT(G/K)$ is a module over $A$ for pointwise 
operations.  We recognize $\cT(G/K)$ as just the induced 
bundle for the representation $\Ad$ restricted to $K$ on 
$\fm$.  

We let $\fm'$ denote the vector-space dual of $\fm$, but
we will also view $\fm'$ as the subspace of $\fg'$ consisting of
linear functionals on $\fg$ that take value 0 on $\fk$, 
so that it is $\fk^\perp$ in the sense of duality. Note
that since our inner product on $\fg$ is $\Ad$-invariant, 
and $\Ad$ restricted to $K$ carries $\fk$ into 
itself, $\Ad$ restricted to $K$ also carries $\fm$ into 
itself. Consequently, $\Ad'$ restricted to $K$ carries $\fm'$ into 
itself.
Since the fibers of a cotangent bundle are just the vector-space
duals of the fibers of the tangent bundle, it is appropriate for
us to set:

\begin{notation} We describe the cotangent bundle, $\O(G/K)$,
of $G/K$ by:
\[
\O(G/K) =  \{\o \in C^{\infty}(G,\fm'): \o(xs) = 
\Ad_{s^{-1}}'(\o(x)) \mbox{ for } x \in G,\ s \in K\} .
\]  
\end{notation} 
\noindent
The pairing between $\cT(G/K)$ and $\O(G/K)$ is given by
\[
\<W, \ \o\>_\cA(x) = \<W(x), \ \o(x)\>
\]
where the pairing on the right is that between $\fm$ and $\fm'$.
It is clear that $\O(G/K)$ is a bimodule over $\cA$ for
``pointwise multiplication''. The differential $d$ from
$\cA$ to $\O(G/K)$ is of course
given by $df(W) = \d_W(f)$.

Our task is to show that, consistent with our general approach
to defining cotangent bundles for actions of $G$ on C*-algebras,
$\O(G/K)$ is generated as a bimodule by the range of the
derivation. This is well-known by the usual methods of 
differential geometry using coordinate charts. We show
here how this works in our setting.
 
\begin{theorem}
With notation as above, the sub-$\cA$-bimodule
of $\O(G/K)$ generated by the range of the
derivation $d$ is $\O(G/K)$ itself.
\end{theorem} 

\begin{proof}
We need to use how the smooth structure on $G/K$
relates to that of $G$. We use the ``slice lemma'', 
lemma 11.21, of \cite{Hll}. Define a function
$\Phi$ from $\fm \times K$ to $G$ by
$\Phi(Y, x) = exp(Y)x$. The slice lemma says that
there is an open neighborhood, $U$ of $0 \in \fm$
such that $\Phi$ restricted to $U\times K$ is a
diffeomorphism onto an open neighborhood
of the identity element, $e$, of $G$. In particular,
each left coset of $K$ 
meets $\Phi(\fm \times \{e\}) = \exp(U)$
in at most one point. Let $C = \exp(U)$. Thus
$C$ is a submanifold of $G$, and is a local cross-section
for the canonical
projection, $p$, of $G$ onto $G/K$. 
Let 
$\cO = p(C)$, so that $p\circ \exp$ restricted
to $U$ is a diffeomorphism from $U$ onto $\cO$
by the definition of the smooth structure on
$G/K$. 

Let $\O_c(\cO)$ be the subspace of $\O(G/K)$
consisting of elements $\o$ of compact support
in $\cO$, that is, such that there is an open
subset $\cO'$ which contains the support
of $\o$ and whose closure $\bar \cO'$ is contained
in $\cO$. Here, for our notation, by $\o$ having
support in $\cO'$ we really mean that as a function on $G$
it has support in $p^{-1}(\cO')$. Notice that $\o$
is entirely determined by its restriction to $C$ (since
it takes value 0 outside of $CK$).

The pull-back, $\tilde \o$,
of $\o$ by $\exp$ is a smooth function from
$U$ into $\fm'$, and is thus a differential form
on $U$. Let $b_1, \cdots, b_m$ be the basis
for $\fm'$ dual to our
basis $X_1, \cdots, X_m$ for $\fm$. As
functions on $\fm$ they
are the coordinate functions. Then
\[
\tilde \o = \sum \tilde g_j db_j
\]
for certain smooth functions $\tilde g_j$
that are supported in $U'$, where $U'$
is the preimage of $\cO'$ under $p\circ \exp$.
Each $db_j$ is just the constant function
with value $b_j \in \fm'$.
Since $U'$ has compact closure in $U$,
we can find a smooth function, $h$, on $U$
that takes value 1 on $U'$ but has compact
support inside $U$. For each $j$ set 
$\tilde h_j = h\tilde b_j$. Then
$\tilde g_j d\tilde h_j = \tilde g_j d\tilde b_j$
so that 
\[
\tilde \o = \sum \tilde g_j d\tilde h_j
\]
while each
$\tilde h_j$ has compact support in $U$.
For each $j$ let $g_j$ and $h_j$ be the 
pullbacks of $\tilde g_j$ and $\tilde h_j$
to $C$ by the inverse of $\exp$. Then
we have
\[
\o = \sum  g_j dh_j.
\]
on $C$.
Extending $g_j$ and $h_j$ to $CK$
and then to functions on $G$ that
are in $C^\infty(G/K)$, we see that
\[
\o = \sum  g_j dh_j
\]
on $G/K$. Thus $\o$ is in the bimodule
generated by the range of $d$.

Since $G/K$ is compact, it can be covered
by a finite number of translates of $\O$.
By use of a partition of the identity subordinate
to such a cover, it follows easily that
the bimodule generated by the range of
$d$ is all of $\cO(G/K)$.
\end{proof} 






\providecommand{\bysame}{\leavevmode\hbox to3em{\hrulefill}\thinspace}

\providecommand{\MR}{\relax\ifhmode\unskip\space\fi MR }
\providecommand{\MRhref}[2]{%
  \href{http://www.ams.org/mathscinet-getitem?mr=#1}{#2}
}
\providecommand{\href}[2]{#2}

}    

\end{document}